\documentclass[10pt]{article}
\usepackage[british]{babel}
\usepackage{mathrsfs}
\usepackage{amsmath}
\usepackage{amssymb}
\usepackage{amsthm}
\usepackage{mathrsfs}
\usepackage{graphics}
\usepackage{enumerate}
\usepackage{epic}
\usepackage{color}
\usepackage[cmtip,arrow]{xy}
\usepackage{pb-diagram,pb-xy}
\usepackage{pict2e}

\newtheorem{lem}{Lemma}[section]
\newtheorem{prop}[lem]{Proposition}
\newtheorem{cor}[lem]{Corollary}
\newtheorem{teorema}[lem]{Theorem}
\newtheorem{remark}{Remark}
\newtheorem{result}{Result}

\newcommand{\F}{\mathbb{F}_q}
\newcommand{\FF}{\mathbb{F}_{q^2}}
\newcommand{\cC}{\mathcal{C}}
\newcommand{\cD}{\mathcal{D}}

\newcommand{\cU}{\mathcal{U}}
\newcommand{\cX}{\mathcal{X}}
\newcommand{\cY}{\mathcal{Y}}
\newcommand{\cZ}{\mathcal{Z}}
\newcommand{\K}{\mathbb{K}}
\def\ord{\text{ord}}
\def\aut{\mbox{\rm Aut}}

\def\bx{\bar{x}}
\def\by{\bar{y}}
\def\bz{\bar{z}}
\def\Fc{\mathbb{F}_{q}\setminus\lbrace 0 \rbrace}
\def\gg{\mathfrak{g}}

\title{On the Dickson-Guralnick-Zieve curve}
\author{Massimo Giulietti\footnote{Massimo Giulietti: massimo.giulietti@unipg.it 
Dipartimento di Matematica ed Informatica,- Universit\`{a} di Perugia- Via Vanvitelli - 60123 Perugia (Italy).}
\newline \and
G\'abor Korchm\'aros, \footnote{G\'abor Korchm\'aros: gabor.korchmaros@unibas.it Dipartimento di Matematica, Informatica ed Economia - Universit\`{a} degli Studi della Basilicata - Viale dell'Ateneo Lucano 10 - 85100 Potenza (Italy)}
\and
Marco Timpanella \footnote{Marco Timpanella: marco.timpanella@unibas.it
Dipartimento di Matematica, Informatica ed Economia - Universit\`{a} degli Studi della Basilicata - Viale dell'Ateneo Lucano 10 - 85100 Potenza (Italy)}}
\date{}
\begin{document}
\maketitle

\vspace{\baselineskip}


  \begin{abstract} The Dickson-Guralnick-Zieve curve, briefly DGZ curve, defined over the finite field $\mathbb{F}_q$ arises naturally from the classical Dickson invariant of the projective linear group $PGL(3,\mathbb{F}_q)$. The DGZ curve is an (absolutely irreducible, singular) plane curve of degree $q^3-q^2$ and genus $\frac{1}{2}q(q-1)(q^3-2q-2)+1.$ In this paper we show that the DGZ curve has several remarkable features, those appearing most interesting are: the DGZ curve has a large automorphism group compared to its genus albeit its Hasse-Witt invariant is positive; the Fermat curve of degree $q-1$ is a quotient curve of the DGZ curve; among the plane curves with the same degree and genus of the DGZ curve and defined over $\mathbb{F}_{q^3}$, the DGZ curve is optimal with respect the number of its $\mathbb{F}_{q^3}$-rational points.
  \end{abstract}

\section{Introduction}
The classical Dickson invariant of the projective linear group $PGL(3,\mathbb{F}_q)$ with $q=p^h$, $p$ prime, is the (absolutely irreducible) homogeneous polynomial $F(x,y,z)\in \mathbb{F}_q(x,y,z)$ given by
$F(x,y,z)=D_1(x,y,z)/D_2(x,y,z)$ where
$$\begin{array}{ll} D_1(x,y,z)=\begin{vmatrix} x & x^q & x^{q^3} \\ y & y^q & y^{q^3} \\ z & z^q & z^{q^3} \end{vmatrix} , \quad D_2(x,y,z)=\begin{vmatrix} x & x^q & x^{q^2} \\ y & y^q & y^{q^2} \\ z & z^q & z^{q^2} \end{vmatrix};\end{array}$$ see \cite{Dickson}.
In geometric terms, the plane curve $\cC$ of projective equation $F(x,y,z)=0$ has an automorphism group $G\cong PGL(3,\mathbb{F}_q)$. In the early 2000s, Guralnick and Zieve pointed out
that $G$ is quite large compared to the genus $\gg$; more precisely $|G|\approx c\gg^{8/5}$, that is, $8/5$ is the amplitude of $|G|$ (with respect to $\gg$); see \cite{gz,KR}. Among the curves with positive Hasse-Witt invariant, $\cC$ is still the unique known example with an automorphism group whose amplitude is as high as (or higher than) $8/5$.
For $q=p$, $\cC$ is ordinary and in this case the amplitude appears exceptionally high, as $8/5$ is far away from the maximum amplitude $3/2$ that a solvable automorphism group of an ordinary curve may have; see \cite{KM}.

In this paper we call $\cC$ the Dickson-Guralnick-Zieve curve, briefly DGZ curve, and show several properties concerning its automorphisms, quotient curves, and the number of its points. In the smallest case $q=2$, the DGZ curve is isomorphic over $\mathbb{F}_8$ to the well known Klein quartic; see Remark \ref{remA16mar2018}. From now we assume $q\ge 3$.

In Section \ref{secDGZ} we show that $\cC$ is an absolutely irreducible plane curve of degree $d=q^3-q^2$. We prove that the singular points of $\cC$ are exactly the points of $PG(2,\mathbb{F}_{q^2})\setminus PG(2,\mathbb{F}_{q})$. Each of them is a ($q-1$)-fold point and it is the center of a unique branch of $\cC$. This shows that there is a one-to-one correspondence between the points of $\cC$ and those of a nonsingular model $\cX$ of $\cC$. In particular, for any $i\geq 1$, the number of points of $\cC$ in $P(2,\mathbb{F}_{q^i})$  equals $|\cX(\mathbb{F}_{q^i})|$ where $\cX(\mathbb{F}_{q^i})$ is the set of all $\mathbb{F}_{q^i}$-rational points of $\cX$. We also find the exact value of $|\cX(\mathbb{F}_{q^i})|$ for $i=1,2,3$,
which are $0,q^4-q, q^6-q^5-q^4+q^3$, but it remains open the problem to compute $|\cX(\mathbb{F}_{q^i})|$ for $i\geq 4$, and more general, the $L$-polynomial of the DGZ curve.

In Section \ref{secaut}, we look inside the action and the ramification groups of a Sylow $p$-subgroup of $G$. The results collected are used in Section \ref{secquotient} to
find the genus of $\cC$ which is $\gg=\frac{1}{2}q(q-1)(q^3-2q-2)+1$, and
show that the quotient curve of $\cC$ arising from a Sylow $p$-subgroup of $G$ is isomorphic to the Fermat curve $\mathcal{F}_{q-1}$ of equation $x^{q-1}+y^{q-1}+z^{q-1}=0$.

In Section \ref{secfullaut} we prove that $G$ is the full automorphism group of the DGZ-curve. Our proof does not depend on deeper results from Group theory, as it only uses the Hurwitz genus formula and the classification of maximal subgroups of $PSL(3,q)$ due to Mitchell \cite{mitchell1911} for odd $q$, and to Hartley \cite{hartley1925} for even $q$.

In Section \ref{sectionnc} we work out the geometry of the DGZ curve. We show that $\cC$ is non-classical in the sense that each nonsingular point of $\cC$ is a flex. We also show that $\cC$ is  $\mathbb{F}_q$-Frobenius non-classical, that is, the tangent at any nonsingular point of $\cC$ contains the image of the point by the $\mathbb{F}_q$-Frobenius map $\Phi_q$ which takes the point
$P=(a:b:c)$ to $P^q=(a^q:b^q:c^q)$. Frobenius non-classicality of $\cC$ remains valid when $q$ is replaced by $q^3$.
Furthermore, we compute the orders $v_P(R)$ and $v_P(S)$ of every $P\in \cC$, for both the ramification divisor $R$ and the St\"ohr-Voloch divisor $S$. This allows us to answer positively the natural question whether $\cC$ has many $\mathbb{F}_{q^3}$-rational points compared to other plane curves defined over $\mathbb{F}_{q^3}$ with the same degree and genus of $\cC$. For this purpose,
the term of $(d,\gg,i)$-optimal curve is useful for curves whose number of $\mathbb{F}_{q^i}$-rational points is big as possible, in the family of all (absolutely irreducible) plane curves of degree $d$ and genus $\gg$ defined over $\mathbb{F}_{q^i}$. With this terminology, the affirmative answer is given by Corollary \ref{cor24mar2018} which states that DGZ-curve is indeed $(q^3-q^2,\frac{1}{2}q(q-1)(q^3-2q-2)+1),3)$-optimal. The above question also has a combinatorial analog where one takes all $\mathbb{F}_{q^i}$-points of the curve, that is, the (possibly singular) points of the curve lying in $PG(2,\mathbb{F}_{q^i})$. Here again, the DGZ-curve is $(q^3-q^2,\frac{1}{2}q(q-1)(q^3-2q-2)+1),3)$-optimal.

\section{Background on automorphisms of curves}
In this section, $\cX$ stands for a (projective, nonsingular, geometrically irreducible, algebraic) curve defined  over an algebraically closed field $\K$ of positive characteristic $p$. Since we work with plane curves, we consider $\cX$ as a nonsingular model of an absolutely irreducible plane curve $\cC$ defined over $\K$. Doing so, we set $\gg(\cC)=\gg(\cX)$ for the genus of $\cX$, $\K(\cX)=\K(\cC)$ for the function field of $\cX$, and $\aut(\cX)=\aut(\cC)$ for the automorphism group of  $\cX$ which fixes $\K$ element-wise.

By a result due to Schmid, $\aut(\cX)$ is finite; see \cite[Theorem 11.50]{HKT}. For a subgroup $G$ of $\aut(\cX)$, let $\bar \cX$ denote a nonsingular model of $\K(\cX)^G$, that is,
a projective nonsingular geometrically irreducible algebraic
curve with function field $\K(\cX)^G$, where $\K(\cX)^G$ consists of all elements of $\K(\cX)$
fixed by every element in $G$. Usually, $\bar \cX$ is called the
quotient curve of $\cX$ by $G$ and denoted by $\cX/G$. The field extension $\K(\cX)|\K(\cX)^G$ is  Galois of degree $|G|$.

Since our approach is mostly group theoretical, we give interpretation of concepts from Function field theory in terms of Group theory.

Let $\Phi$ be the cover of $\cX\mapsto \bar{\cX}$ where $\bar{\cX}=\cX/G$ is a quotient curve of $\cX$ with respect to $G$.
 A point $P\in\cX$ is a ramification point of $G$ if the stabilizer $G_P$ of $P$ in $G$ is nontrivial; the ramification index $e_P$ is $|G_P|$. A point $\bar{Q}\in\bar{\cX}$ is a branch point of $G$ if there is a ramification point $P\in \cX$ such that $\Phi(P)=\bar{Q}$; the ramification (branch) locus of $G$ is the set of all ramification (branch) points. The $G$-orbit of $P\in \cX$ is the subset of $\cX$
$o=\{R\mid R=g(P),\, g\in G\}$, and it is {\em long} if $|o|=|G|$, otherwise $o(P)$ is {\em short}. For a point $\bar{Q}$, the $G$-orbit $o$ lying over $\bar{Q}$ consists of all points $P\in\cX$ such that $\Phi(P)=\bar{Q}$. If $P\in o$ then $|o|=|G|/|G_P|$ and hence $\bar{Q}$ is a branch point if and only if $o$ is a short $G$-orbit. It may be that $G$ has no short orbits. This is the case if and only if every nontrivial element in $G$ is fixed--point-free on $\cX$, that is, the cover $\Phi$ is unramified. On the other hand, $G$ has a finite number of short orbits. For a non-negative integer $i$, the $i$-th ramification group of $\cX$
at $P$ is denoted by $G_P^{(i)}$ (or $G_i(P)$ as in \cite[Chapter
IV]{serre1979})  and defined to be
$$G_P^{(i)}=\{g\mid \ord_P(g(t)-t)\geq i+1, g\in
G_P\}, $$ where $t$ is a uniformizing element (local parameter) at
$P$. Here $G_P^{(0)}=G_P$.
The structure of $G_P$ is well known; see for instance \cite[Chapter IV, Corollary 4]{serre1979} or \cite[Theorem 11.49]{HKT}.
\begin{result}
\label{res74} The stabilizer $G_P$ of a point $P\in \cX$ in $G$ has the following properties.
\begin{itemize}
\item[\rm(i)] $G_P^{(1)}$ is the unique normal $p$-subgroup of $G_P$;
\item[\rm(ii)] For $i\ge 1$, $G_P^{(i)}$ is a normal subgroup of $G_P$ and the quotient group $G_P^{(i)}/G_P^{(i+1)}$ is an elementary abelian $p$-group.
\item[\rm(iii)] $G_P=G_P^{(1)}\rtimes U$ where the complement $U$ is a cyclic whose order is prime to $p$.
\end{itemize}
\begin{result}
\label{resnaka}
Let $G$ be a subgroup of $\aut(\cX)$. For $P\in \cX$ put $e=|G_P/G_P^{(1)}|$ and $d=|G_P^{(1)}/G_P^{(2)}|$. Then $e$ divides $d-1$.
\end{result}

\end{result}
Let $\bar{\gg}$ be the genus of the quotient curve $\bar{\cX}=\cX/G$. The Hurwitz
genus formula  gives the following equation
    \begin{equation}
    \label{eq1}
2\gg-2=|G|(2\bar{\gg}-2)+\sum_{P\in \cX} d_P.
    \end{equation}
    where
\begin{equation}
\label{eq1bis}
d_P= \sum_{i\geq 0}(|G_P^{(i)}|-1).
\end{equation}
Here $D(\cX|\bar{\cX})=\sum_{P\in\cX}d_P$ is the {\emph{different}}. For a tame subgroup $G$ of $\aut(\cX)$, that is for $p\nmid |G_P|$,
$$\sum_{P\in \cX} d_P=\sum_{i=1}^m (|G|-\ell_i)$$
where $\ell_1,\ldots,\ell_m$ are the sizes of the short orbits of $G$.

A subgroup of $\aut(\cX)$ is a $p'$-group (or a prime to $p$) group if its order is prime to $p$. A subgroup $G$ of $\aut(\cX)$ is {\em{tame}} if the $1$-point stabilizer of any point in $G$ is $p'$-group. Otherwise, $G$ is {\em{non-tame}} (or {\em{wild}}). Obviously, every $p'$-subgroup of $\aut(\cX)$ is tame, but the converse is not always true. From the classical Hurwitz's bound,
if $|G|>84(\gg(\cX)-1)$ then $G$ is non-tame; see  \cite{stichtenoth1973II} or \cite[Theorems 11.56]{HKT}:
An orbit $o$ of $G$ is {\em{tame}} if $G_P$ is a $p'$-group for $P\in o$, otherwise $o$ is a {\em{non-tame orbit}} of $G$.

Stichtenoth's result  \cite{stichtenoth1973II} on the number of short orbits of large automorphism groups; see \cite[Theorems 11.56, 11.116]{HKT}:
\begin{result}
\label{res56.116} Let $G$ be a subgroup of $\aut(\cX)$ whose order exceeds $84(\gg(\cX)-1)$. Then $G$ has at most three short orbits, as follows:
\begin{itemize}
\item[\rm(a)] exactly three short orbits$,$ two tame and
one non-tame$,$ and $|G|< 24\gg(\cX)^2$;

\item[\rm(b)] exactly two short orbits$,$ both non-tame$,$ and  $|G|<16 \gg(\cX)^2$;

\item[\rm(c)] only one short orbit which is non-tame$;$

\item[\rm(d)] exactly two short orbits$,$ one tame and one non-tame.
\end{itemize}
\end{result}
Nakajima's bound \cite{Nakajima}; see also  \cite[Theorem 11.54]{HKT}:
\begin{result}
\label{resnaga} If $\cX$ has positive $p$-rank and $S$ is a $p$-subgroup of $\aut(\cX)$ then
\begin{equation}
\label{naka16feb2013}
|S|\leq \left\{
\begin{array}{lll}
\textstyle\frac{p}{p-2}\,(\gg(\cX)-1)\quad {\mbox{for}}\quad \gamma(\cX)>1,\\
\quad\quad\,\, \gg(\cX)-1\quad\,\, {\mbox{for}}\quad \gamma(\cX)=1.
\end{array}
\right.
\end{equation}
\end{result}
The following corollary to the Deuring Shafarevic formula; see \cite[Theorem 11.129]{HKT}:
\begin{result}
\label{lem29dic2015} If $\cX$ has zero $p$-rank then any element of order $p$ has exactly one fixed point $P$.
\end{result}
 The results from Group theory which play a role in the paper are quoted below. Here $G$ stands for any finite group. We use standard notation and terminology; see \cite{Machi}.

The orbit theorem \cite[Theorem 3.2]{Machi}:
\begin{result}
\label{resorbit}
Let $G\leq \aut(\cX)$ and $P\in \cX$. Then $$|G|=|G_P||P^G|.$$
\end{result}
\begin{result}
\label{resmachi}
Let $G$ be a $p$-group, and $H$ a proper subgroup of $G$. Then $H$ is properly contained in its normalizer.
\end{result}
The maximal subgroups of $PGL(3,q)$ were classified by Mitchell \cite{mitchell1911} and \cite{hartley1925}; see also \cite[Theorem A.10]{HKT}. In this paper
we only need the following corollaries of that classification; see \cite[Theorem 29]{mitchell1911}, \cite[pg. 157]{hartley1925}.
\begin{result}
\label{dicksonclassification}
For $q=p^m,$  the following is
a complete list of subgroups of the group $PGL(2,q)$ up to conjugacy$:$
\begin{enumerate}
\item[\rm(i)] the cyclic group of order $n$ with $n\mid(p^m\pm 1);$
\item[\rm(ii)] the elementary abelian $p$-group of order $p^f$ with $f\leq m;$
\item[\rm(iii)] the dihedral group of order $2n$ with $n\mid
(q\pm 1);$
\item[\rm(iv)] the alternating group of degree $4$ for $p>2$, or $p=2$ and $m$
even$;$
\item[\rm(v)] the symmetric group of degree $4$ for $p>2;$
\item[\rm(vi)] the alternating group of degree $5$ for $5\mid(q^2-1);$
\item[\rm(vii)] the semidirect product of an elementary abelian
$p$-group of order $p^h$ by a cyclic group of order $n$ with
$h\leq m$ and $n\mid(q-1);$
\item[\rm(viii)] $PSL(2,p^f)$ for $f\mid m;$
\item[\rm(ix)] $PGL(2,p^f)$ for $f\mid m$.
\end{enumerate}
\end{result}
\begin{result}
\label{mitchellclassification}
\index{Mitchell Theorem} \index{Theorem!Mitchell}
For $q=p^k,$  the following is
a complete list of subgroups of the group $PSL(3,q)$ up to conjugacy$:$
\begin{enumerate}
\item[\rm(i)] groups of order $q^3(q-1)^2(q+1)/\mu;$
\item[\rm(ii)] groups of order $6(q-1)^2/\mu$;
\item[\rm(iii)] groups of order $3(q^2+q+1)/\mu;$
\item[\rm(iv)] groups of order $q(q+1)(q-1)$;
\item[\rm(v)] $PSL(3,p^m)$, where $m$ is a factor of $k$;
\item[\rm(vi)] groups containing $PSL(3,p^m)$ as self-conjugate subgroups of index $3$ if $p^m-1$ is divisible by $3$ and $k/m$ is divisible by $3$;
\item[\rm(vii)] the group $PSU(3,p^{2m})$, where $2m$ is a factor of $k$;
\item[\rm(viii)] groups containing $PSU(3,p^{2m})$ as self-conjugate subgroups of index $3$ if $p^m+1$ is divisible by $3$ and $k/2m$ is divisible by $3$;
\item[\rm(ix)] The Hessian groups of order $216$ if $q-1$ is divisible by $9$, $72$ and $36$ if $q-1$ is divisible by $3$.
\item[\rm(x)] Groups of order $168$, which exist if $\sqrt{-7}$ exists in $\F$.
\item[\rm(xi)] Groups of order $360$, which exist if both $\sqrt{5}$ and a cube root of unity exist in $\F$;
\item[\rm(xii)] Groups of order $720$ containing the groups of order $360$ as self-conjugate
subgroups. These exist only for $p = 5$ and $k $even;
\item[\rm(xiii)] Groups of order $2520$, each isomorphic with the alternating group of degree $7$. These exist only for $p = 5$ and $k$ even.
\end{enumerate}
\end{result}
\begin{result}
\label{resmitchell}
Let $\Omega$ be a set of smallest size on which $PSL(3,q)$ has a nontrivial action. Then $|\Omega|\geq q^2+q+1$.
\end{result}
\section{Background on non-classical plane curves}
An irreducible (not necessarily nonsingular) plane curve $\cC$ defined over $\K$ is called {\em{non-classical}} if its Hessian curve vanishes; see \cite{SV} and
\cite[Section 7.8]{HKT}. If $\cC$ is given by a homogeneous equation $F(x,y,z)=0$,  a necessary and sufficient condition for $\cC$ to be non-classical is the existence of homogeneous polynomials $G_0(x,y,z),G_1(x,y,z), G_2(x,y,z)$ of the same degree together with a homogeneous polynomial $H(x,y,z)$ such that
\begin{equation}
\label{eq19mar2018}
HF=G_1^{p^m}x+G_2^{p^m}y+G_0^{p^m}z
\end{equation}
for some $m\ge 1$.
Let $\mathcal{L}$ be the (not necessary complete) linear series cut out by lines. For a point $P\in \cC,$ the $({\mathcal{L}},P)$-order sequence is $(j_0,j_1,j_2)$ with $j_0=0<j_1<j_2$ where
$j_i=I(P,\cC\cap r)$ is the intersection number of $\cC$ with a line $r$ through $P$ and $i=1$ or $2$ according as $r$ is a non-tangent line or the tangent to $\cC$ at $P$.
The $\mathcal{L}$-order sequence of $\cC$ is $(\varepsilon_0,\varepsilon_1,\varepsilon_2)$ if $(\varepsilon_0,\varepsilon_1,\varepsilon_2)=(j_0,j_1,j_2)$ for all but a finite number of points of $\cC$. Let $x$ be a separable variable of $\cC$, and let $D^{(i)}(t)$ be the $i$-th Hasse derivative of $t\in \K(\cC)$ relative to $x$.
If $\cC$ is non-classical and (\ref{eq19mar2018}) holds then its order sequence is $(0,1,p^m)$ and the Wronskian of $\cC$ with respect to $\mathcal{L}$ is the determinant
$$W_R=\left|
        \begin{array}{ccc}
          x & y & 1 \\
          D(x) & D(y) & 0 \\
          D^{(p^m)}(x) & D^{(p^m)}(y) & 0 \\
        \end{array}
      \right|=D^{(p^m)}(y).
$$
If $V$ is the intersection divisor of $\cC$ with a line of $PG(2,\K)$, the ramification divisor of $\mathcal{L}$ is
$$R={\rm{div}}(W_R)+(p^m+1){\rm{div}}(dx)+3V,$$
and $\deg(R)=(1+p^m)(2\gg(\cX)-2)+3 \deg(\cC)$.
Let $v_P(R)={\rm{ord}}_P(W)$. Then $R=\sum v_P(R)P$.

If the irreducible plane curve $\cC$ is defined over a finite field $\mathbb{F}_q$ (and $\K=\overline{\mathbb{F}}_q$), another concept of non-classicality is defined, namely that arising from the action of the $q$-Frobenius map $\Phi$ on the nonsingular points of $\cC$ defined by $\Phi(P)=P^q$ where $P=(x:y:z)$ and $P^q=(x^q:y^q:z^q)$; see \cite{SV} and \cite[Section 8.6]{HKT}. More precisely, $\cC$ is called {\emph{$\mathbb{F}_q$-Frobenius non-classical}}
if the tangent at any nonsingular point $P\in \cC$ contains $P^q$. A necessary and sufficient condition for a non-classical curve $\cC$ with $1<p^m\leq q$ to be $\mathbb{F}_q$-Frobenius non-classical is the existence of a homogeneous polynomial $T(x,y,z)$ such that
\begin{equation}
\label{eqA19mar2018}
TF=G_1x^{q/p^m}+G_2y^{q/p^m}+G_0z^{q/p^m}
\end{equation}
with $G_0,G_1,G_2$ as given in (\ref{eq19mar2018}); see \cite[Theorem 8.72]{HKT}.
Frobenius non-classical curves are somewhat rare; see \cite{borgescon,SV}. In some cases, they have many points over $\mathbb{F}_q$; see \cite{AB,borgeshomma,hk1999,SV}. Also, they are closely related to univariate polynomials with minimal values sets, see \cite{borges1}.
Since $\mathcal{L}$ is defined over $\mathbb{F}_q$, $\cC$ also has its $\mathbb{F}_q$-Frobenius order sequence $(\nu_0,\nu_1)$; see \cite[Definition 8.46]{HKT}. If (\ref{eqA19mar2018}) holds then $\nu_0=0$ and $\nu_1=p^m$ by \cite[Proposition 8.42]{HKT}. Let
$$W_S=\left|
        \begin{array}{ccc}
          x^q & y^q & 1 \\
          x & y & 1 \\
          D^{(p^m)}(x) & D^{(p^m)}(y) & 0 \\
        \end{array}
      \right|=(x-x^q)D^{(p^m)}(y).
$$
If $V$ is the intersection divisor of $\cC$ with a line of $PG(2,\mathbb{F}_q)$,
the St\"ohr-Voloch divisor of $\mathcal{L}$ over $\mathbb{F}_q$ is
 $$S={\rm{div}}(W_S)+p^m{\rm{div}}(dx)+(q+2)V,$$
and $\deg(S)=p^m(2\gg(\cC)-2)+(q+2)\deg(\cC)$; see \cite[Definition 8.45]{HKT}.

\section{The DGZ-curve and its singular points}
\label{secDGZ}
A straightforward computation shows that both $D_1(x,y,z)$ and $D_2(x,y,z)$ are $GL(3,\F)$ invariant up to a constant. In other words, the following result dating back to Dickson holds.
\begin{lem}[Dickson]
\label{lemA30gen2018} Let $A\in GL(3,\F)$, and $[x,y,z]^t=A[\bx,\by,\bz]^t$. Then $D_1(x,y,z)=\det(A)D_1(\bx,\by,\bz)$, and $D_2(x,y,z)=\det(A)D_2(\bx,\by,\bz)$.
\end{lem}
As a corollary of Lemma \ref{lemA30gen2018}, the rational function
\begin{equation}
\label{eqA30gen2018}
F(x,y,z)=\frac{D_1(x,y,z)}{D_2(x,y,z)}
\end{equation}
is $GL(3,\F)$ invariant. 
\begin{lem}
\label{lemB0gen2018} $F(x,y,z)$ is a homogeneous polynomial of degree $q^3-q^2$ defined over $\F$.
\end{lem}
\begin{proof}
From Lemma \ref{lemA30gen2018}, the algebraic plane curve $\mathcal{D}_1$ with homogeneous equation
$D_1(x,y,z)=0$ is left invariant by $PGL(3,\F)$, and the same holds for the algebraic plane curve $\mathcal{D}_2$ with homogeneous equation $D_2(x,y,z)=0$. Obviously, the line $\ell$ of equation $x=0$ is a component of $\mathcal{D}_2$.    Since $PGL(3,\F)$ acts on the set of all lines of the projective plane $PG(2,\F)$ as transitive permutation group, this yields that each such line is a component of $\mathcal{D}_2$. On the other hand, $\deg\,D_2(x,y,z)=q^2+q+1$ and $PG(2,\F)$ has as many as $q^2+q+1$ lines. Therefore, $\mathcal{D}_2$ splits into $q^2+q+1$ lines each counted with multiplicity $1$. The same argument applies to $\mathcal{D}_1$ showing that each line of $PG(2,\F)$ is a component of $\mathcal{D}_1$. Therefore $D_2(x,y,z)$ divides $D_1(x,y,z)$.
\end{proof}
From now on $\cC$ stands for the algebraic plane curve with homogenous equation $F(x,y,z)=0$. According to Introduction, $\cC$ is the DGZ curve.
\begin{remark}
\label{remA16mar2018}
{\em{Let $q=2$. A straightforward computation shows that
\begin{equation}
\label{eqC19mar2018}
{\mbox{$F(x,y,z)=x^4 + x^2y^2 + x^2yz + x^2z^2 + xy^2z + xyz^2 + y^4 + y^2z^2 + z^4$}}
\end{equation}
This curve already mentioned in Serre's lecture notes \cite{serre1985} was investigated by Top \cite{jt}. Actually $\cC$ is isomorphic over $\mathbb{F}_8$ to the well known Klein curve of equation $x^3y+y^3z+z^3x=0$. This implies that $\aut(\cC)\cong PGL(3,2)$.}}
\end{remark}
\begin{prop}
\label{propA30gen}
The DGZ curve  has no $\mathbb{F}_q$-rational points.
\end{prop}
\begin{proof} Since $PGL(3,\F)$ acts on the set of all points of $PG(2,\F)$ as a transitive permutation group, it is sufficient to show the result for the origin $O=(0:0:1)$. Since $D_1(x,y,1)=xy^q-yx^q+g_1(x,y)$ with $\deg g_1(x,y)>q+1$ and $D_2(x,y,1)=xy^q-x^qy+g_2(x,y)$ with $\deg g_2(x,y)>q+1$, the origin is a ($q+1$)-fold point for both curves $\mathcal{D}_1$ and $\mathcal{D}_2$. From this the result follows.
\end{proof}
\begin{prop}
\label{propB30gen}
Each point lying in $PG(2,\mathbb{F}_{q^2})\setminus PG(2,\F)$ is a $(q-1)$-fold point of $\cC$.
\end{prop}
\begin{proof} Since $PGL(3,\F)$ acts on the set of all points of $PG(2,\FF)\setminus PG(2,\F)$ as a transitive permutation group, it suffices to perform the proof for a point $P=(\alpha:0:1)$ with $\alpha\in \FF\setminus \F$.
In the affine plane $AG(2,\FF)$ with line at infinity $z=0$, the translation $\tau :(x,y)\mapsto (x-\alpha,y)$ taking $P$ to the origin $O=(0,0)$, maps $\mathcal{D}_1$ and $\mathcal{D}_2$ to the curves $\mathcal{Y}_1$ and $\mathcal{Y}_2$ with affine equations
$$H_1(x,y)=\begin{vmatrix} (x+\alpha) & (x+\alpha)^q & (x+\alpha)^{q^3} \\ y & y^q & y^{q^3} \\ 1 & 1 & 1 \end{vmatrix} $$ and $$H_2(x,y)=\begin{vmatrix} (x+\alpha) & (x+\alpha)^q & (x+\alpha)^{q^2} \\ y & y^q & y^{q^2} \\ 1 & 1 & 1 \end{vmatrix}$$
respectively. Expanding along the first row yields $H_1(x,y)=(\alpha-\alpha^q)y^q+g_1(x,y)$ and $H_2(x,y)=(\alpha-\alpha^q)y+g_2(x,y)$, where $\deg g_1(x,y)>q$ and $\deg g_2(x,y)>1$.
The translation $\tau$ maps $\cC$ to a curve $\cZ$ with affine equation $G(x,y)$. Then
\begin{equation}
\label{zeta}
G(x,y)=\frac{H_1(x,y)}{H_2(x,y)}=y^{q-1}+G_1(x,y),
\end{equation} with $\deg G_1(x,y)>q-1$. This shows that $P$ is a $(q-1)$-fold point of $\cC$.
\end{proof}
\begin{lem}
\label{lem2Feb}
The singular points of the DGZ curve are those which lie in $PG(2,\mathbb{F}_{q^2})\setminus PG(2,\F)$ .
\end{lem}
\begin{proof}
Let $P$ be a singular point of $\cC$. Since $D_1(x,y,z)=D_2(x,y,z)F(x,y,z)$,
$P$ is a singular point of $\mathcal{D}_1$, as well. On the other hand, a straightforward computation yields $$ \dfrac{\partial D_1}{\partial x}=y^q-y^{q^3}, \quad
\dfrac{\partial D_1}{\partial y}=x^{q^3}-x^q.$$ Therefore $\mathcal{D}_1$, and hence $\cC$,  can only have singularities at points lying in $PG(2,\FF)$. Now, the assertion follows from Proposition \ref{propB30gen}.
\end{proof}
\begin{lem}
\label{lem31jan2018}
Each point of the DGZ curve lying in $PG(2,\FF)\setminus PG(2,\F)$ is the center of a unique branch. Furthermore, such a branch has order $q-1$, and its intersection number with its tangent is $q$.
\end{lem}
\begin{proof}
Let $P\in PG(2,\FF)\setminus PG(2,\F)$ be a point of $\cC$. W.l.o.g. $P=(\alpha:0:1)$. With the notation used in the proof of Proposition \ref{propB30gen}, take a branch $\gamma$ of $\cZ$ centred in $(0,0)$. From the proof of that proposition, $\gamma$ has a primitive representation $(x(t),y(t))$ with $x(t),y(t)\in \overline{\mathbb{F}}_q[[t]]$ where $$\begin{cases} x(t)=c_1t^u+... \\ y(t)=d_1t^v+... \end{cases} $$ with $c_1\neq 0$, $d_1 \neq 0$, and,  by \eqref{zeta}, $v>u$ and $u\leq (q-1)$. By direct computation $H_1(x(t),y(t))=(\alpha-\alpha^q)d_1^qt^{vq}+d_1c_1^qt^{v+qu}+g(t)$, where all the terms in $g(t)$ have higher degree than $(v+qu)$. On the other hand $G(x(t),y(t))=0$. Therefore $vq=v+qu$ whence $u=q-1$ and $v=q$. From \cite[Theorem 4.36]{HKT}, $\gamma$ is the unique branch of $\cZ$ centred at $O$, and its order equals $q-1$. Going back to $\cC$, there is unique branch of $\cC$ centred at $P$ and
 \begin{equation} \label{branchP}
 \gamma = \begin{cases} x(t)=\alpha +c_1t^{q-1}+\ldots \\ y(t)=d_1t^q+\ldots \end{cases}.
 \end{equation}
is a primitive representation for it. Since the tangent at $\gamma$ is the line $y=0$, $I(P,\cC\cap \{y=0\})=I(P,\gamma \cap \{y=0\})=q$.
\end{proof}
\begin{prop}
\label{propA31gen2018}
The DGZ curve is absolutely irreducible.
\end{prop}
\begin{proof}
 First we show that $\cC$ does not have two different absolutely irreducible components. Assume on the contrary that $\mathcal{G}$ and $\mathcal{H}$ are two such components. Let $P\in \mathcal{G}\cap \mathcal{H}$, and take a branch $\gamma$ of $\mathcal{G}$ and a branch $\delta$ of $\mathcal{H}$ both centred at $P$. Then $P$ is a singular point of $\cC$ with at least two branches centred at $P$. From Lemma \ref{lem2Feb}, $P\in PG(2,\FF)\setminus PG(2,\F)$, but this contradicts Lemma \ref{lem31jan2018}. Therefore, $\cC$ has a unique absolutely irreducible component $\mathcal{G}$ with multiplicity $\mu\ge 1$. On the other hand, Lemma \ref{lem2Feb} yields that $\cC$ has only a finite number of singular points.
 Therefore, $\mu=1$.
\end{proof}
As a corollary we have the following result.
\begin{cor}
\label{cor31gen2018}
At every point of the DGZ curve, there is only one branch of $\cC$ centred at that point.
\end{cor}
By Corollary \ref{cor31gen2018}, there is a bijection between the points of $\cX$ and those of $\cC$. This shows that $\cC(\mathbb{F}_{q^n})=\cX(\mathbb{F}_{q^n})$ for every $n\ge 1$.
The point-set of $PG(2,\mathbb{F}_{q^3})$ is
 partitioned in three subsets, $\Lambda_1, \Lambda_2$ and $\Lambda_3$
 where $\Lambda_1$ is the set of all points in $PG(2,\F)$, $\Lambda_2$ consists of all points of  $PG(2,\mathbb F_{q^3})\setminus PG(2,\F) $ covered by the lines of $PG(2,\F)$, and $\Lambda_3$ is the set of remaining points in $PG(2,\mathbb{F}_{q^3})$. Obviously, $|\Lambda_1|=q^2+q+1$. A direct computation shows that $|\Lambda_2|=(q^2+q+1)(q^3-q).$ Hence
\begin{equation}
\label{leA14mar2018}
|\Lambda_3|=q^6-q^5-q^4+q^3.
\end{equation}
\begin{lem}
\label{le14mar2018} $\cC(\mathbb{F}_{q^3})=\Lambda_3$.
\end{lem}
\begin{proof}
From Proposition \ref{propA30gen}, $\cC(\mathbb{F}_{q^3})\cap \Lambda_1=\emptyset$. Furthermore, if $L$ is a line of $PG(2,\F)$ then $L \cap \cC \subseteq \cC(\mathbb{F}_{q^2})$ as a consequence of Proposition \ref{propB30gen} and of the B\'ezout's theorem. Therefore, $\cC(\mathbb{F}_{q^3})\cap \Lambda_2=\emptyset$ also.
Let $P \in \Lambda_3$, with $P=(\alpha: \beta: \gamma)$. Then $$D_1(P)=\begin{vmatrix} \alpha & \alpha^q & \alpha^{q^3} \\ \beta & \beta^q & \beta^{q^3} \\ \gamma & \gamma^q & \gamma^{q^3} \end{vmatrix}=\begin{vmatrix} \alpha & \alpha^q & \alpha \\ \beta & \beta^q & \beta \\ \gamma & \gamma^q & \gamma \end{vmatrix}=0.$$ On the other hand $D_2(P)\neq 0$, hence $F(P)=0$ and the assertion follows.
\end{proof}

\begin{teorema}
The DGZ curve has genus $\frak{g}=\frac{1}{2}q(q-1)(q^3-2q-2)+1$.
\end{teorema}
\begin{proof}
From Equation (\ref{eqA30gen2018}),
\begin{equation*}
\dfrac{\partial F}{\partial y}=\left(\dfrac{\partial D_1}{\partial y}D_2-D_1\dfrac{\partial D_1}{\partial y}\right)\dfrac{1}{D_2^2}.
\end{equation*}
In affine coordinates, $f_y(x,y)=$
\begin{equation*}
\frac{(x^{q^3+q}-x^{q^3+1})(y^{q^2}-y^q)+(x^{q^2+q}-x^{q^2+1})(y^q-y^{q^3})+(x^{2q}-x^{q+1})(y^{q^3}-y^{q^2})}{D_2(x,y,1)^2}.
\end{equation*}
This shows that $f_y$ does not vanish
on the generic points of $\mathcal C$. Hence, $x$ is a separating variable of $\K(\cC)$. Thus $\text{div}(dx)\neq 0$ and $\deg(\text{div}(dx))=2\frak{g(\mathcal{X})}-2$; see \cite[Theorem 5.50]{HKT}.
Let $P$ be a point of $\cC$. Four cases are distinguished according as
\begin{itemize}
\item[(a)] $P=(a:b:1)$ and $a\in \mathbb{F}_{q^2}\setminus \mathbb{F}_q$;
\item[(b)] $P=(a:b:1)$ and $a\in \mathbb{F}_q$;
\item[(c)] $P=(a:b:1)$ and $a\notin \mathbb{F}_{q^2}$;
\item[(d)] $P=(a:b:0)$.
\end{itemize}
In Case (a), Proposition \ref{lem31jan2018} shows that the (unique) branch $\gamma$ centred at $P$ has a primitive representation $$\begin{cases} x(t)=a+c_1t^{q-1} \\ y(t)=b+d_1t^q+\ldots \end{cases} .$$ In particular, $\ord_\gamma(dx)=(q-2)$.

 In Case (b) some more efforts are needed. In this case, the (unique) branch centered at $P$ has a primitive representation $$ \begin{cases} x(t)=a+c_1t^{q} \\ y(t)=b+d_1t^{q-1}+\ldots \end{cases} .$$ Since $D_1(x,y,z)=D_2(x,y,z)F(x,y,z)$, we have that $D_1(x(t),y(t),1)=0$ in $\K[[t]]$. Thus $dD_1(x(t),y(t),1)/dt=0$ in $\K[[t]]$. From this, by a direct computation,  $$x'(t)=\frac{x(t)^qy'(t)-x(t)^{q^3}y'(t)}{y(t)^q-y(t)^{q^3}}=\frac{t^{(q+2)(q-1)}g_1(t)}{t^{q(q-1)}g_2(t)}=t^{2q-2}g(t).$$ Therefore $\ord_\gamma(dx)=2q-2$.

In Case (c), $P$ is a simple point of $\mathcal C$. Again, let $\gamma$ be the unique branch of $\mathcal C$ centered at $P$.
 Therefore, $\ord_\gamma(dx)=0$ unless the tangent line of $\mathcal C$ at $P$ is vertical. That line has equation $x=a$. This yields that the univariate polynomial
$$D(y)=D_1(a,y,1)=\begin{vmatrix} a & a^q & a^{q^3} \\ y & y^q & y^{q^3} \\ 1 & 1 & 1 \end{vmatrix}
$$
has a multiple root. Since $D'(y)=(a-a^{q^2})^q$, this gives $a\in \mathbb F_{q^2}$, a contradiction.

 In Case (d),
 $$d\left(\frac{1}{x(t)}\right)=\frac{1}{x(t)^2}\frac{dx(t)}{dt}=t^{-2}.$$
 Hence $\ord_\gamma(dx)=-2$.

 Summing up this gives  $\deg(dx)=(q-2)q^2(q^2-q)+(2q-2)(q^2-q)q-2(q^2-q)=2\frak{g}(\cC)-2$ whence the formula for $\frak{g}(\cC)$ follows.
\end{proof}

\section{Action of a Sylow $p$-subgroup of $G$ and its normalizer on the DGZ-curve}
\label{secaut}
As we have pointed out after Lemma \ref{lemA30gen2018}, $F(x,y,z)$ is a $GL(3,q)$-invariant homogeneous polynomial.
Therefore, $PGL(3,q)$ is a subgroup of $\aut(\cC)$. In terms of the function field $\K(x,y)$ of $\cC$, this shows that the map $\varphi_{\alpha,\beta}$ defined by $$\varphi_{\alpha,\beta}=\begin{cases} x'=x+\alpha\\y'=y+\beta\end{cases} \quad \alpha,\beta \in \F,$$ is an automorphism of $\K(x,y)$. These automorphisms form the translation group $T$ of $\aut(\cC)$, and the fixed points of $T$ are precisely the $q^2-q$ points of $\cC(\mathbb{F}_{q^2})\cap \ell_\infty$, where $\ell_\infty$ is the line $z=0$.

Let $Q$ be the Sylow $p$-subgroup of $PGL(3,q)$ whose elements have matrix representation
$$\left(\begin{matrix}
1 & 0 & \alpha \\
\gamma & 1 & \beta \\
0 & 0 & 1 \\
\end{matrix}\right)
$$ with $\alpha, \beta, \gamma$ in $\mathbb{F}_q$.
A straightforward computation shows that the elements of $Q$ fix the point $Y_\infty=(0:1:0)$, and those with $\gamma\neq 0$ have no further fixed point in $PG(2,\K)$.

Now look at $Q$ as a subgroup of $\aut(\cC)$. Then $Q$ contains $T$ and the subgroup $\Phi$ consisting of all maps $$\phi_\gamma=\begin{cases} x'=x \\ y'=\gamma x+y \end{cases},$$ with $\gamma \in \F$.
More precisely, $Q=T\rtimes \Phi$ where $|Q|=q^3, |T|=q^2$ and $|\Phi|=q$. Also, the quotient group $\bar{Q}=Q/T$ is an elementary abelian group of order $q$.

Since $Y_\infty\not\in \cC$, it turns out that no element in $Q\setminus T$ fixes a point in $\cC$.

Furthermore, the maps $$\psi_{\lambda,\mu}=\begin{cases} x'=\lambda x\\y'=\mu y\end{cases},$$ where $\lambda, \mu \in \Fc$ are automorphisms of $\cC$ and they form an abelian subgroup $\Psi$ of $\aut(\cC)$ of order $(q-1)^2$. When $\lambda=\mu$, the $q-1$ maps $\psi_{\lambda,\lambda}$ form the dilatation group $D$ of $\aut(\cC)$, and the fixed points of $D$ are the $q^2-q$ points of $\cC(\mathbb{F}_{q^2})\cap \ell_\infty$. Also, the quotient group $\Psi/D$ is a cyclic group of order $q-1$.

A straightforward computation shows that $\Psi$ is contained in the normalizer of $Q$. Therefore, the group generated by them is the semidirect product $Q\rtimes \Psi$. Furthermore, $\Psi$ is also contained in the normalizer of $T$.
Hence the quotient group $(Q\rtimes \Psi)/T$ is isomorphic to the semidirect product $\bar{Q}\rtimes \bar{\Psi}$ where $\bar{\Psi}=(\Psi T)/T$. Observe that $\bar{Q}\rtimes \bar{\Psi}$ can be regarded as an automorphism group of the projective line $\ell_\infty$. Doing so, if $\ell_\infty=\{(1:m:0)| m\in \K\}\cup \{(0:1:0)\}$ then
$\bar{Q}\rtimes \bar{\Psi}$ consists of all maps such that $(1:m:0)\mapsto (1:am+b:0)$ with $a,b$ ranging over $\K$ and $a\neq 0$. Under the action of $\bar{Q}\rtimes \bar{\Psi}$, $\cC(\mathbb{F}_{q^2})\cap \ell_\infty$ splits into $q-1$ orbits $\Delta_1,\ldots, \Delta_{q-1}$ each of size $q$.

\begin{remark}
\label{rem15mar2018}{\em{
Let $S_p$ be a Sylow $p$-subgroup of $\aut(\cC)$ containing $Q$. Assume that $S_p$ is larger than $Q$. From Result \ref{resmachi}, there exists a subgroup $U\leq S_p$ such that $Q\trianglelefteq  U$ and that $[U:Q]=p^r$ for some $r\geq 1$. We show that $T\trianglelefteq  U$. For any $u\in U$ and $t\in T$, the conjugate $t'=utu^{-1}$ of $t$ by $u$ has $q^2-q$ fixed points. On the other hand, no element in $Q\setminus T$ has a fixed point in $\cC$. Therefore, $t'\in T$, and hence
$T\trianglelefteq  U$. It turns out that $U$ leaves $\cC(\mathbb{F}_{q^2})\cap \ell_\infty$ invariant. Let $\hat{U}=U/T$, $\hat{Q}=Q/T$ and $\bar{U}=U/Q$. Then
$\hat{U}$ is the permutation group induced by $U$ on $\cC(\mathbb{F}_{q^2})\cap \ell_\infty$. From $\hat{Q}\trianglelefteq  \hat{U}$, the $Q$-orbit partition $\{\Delta_1,\ldots, \Delta_{q-1}\}$ is left invariant by $\hat{U}$. Since this partition has as many as $q-1$ members while $p$ divides $|\hat{U}|$, this yields that $\hat{U}$ must fixes at least two members. In other words, $\bar{U}$ fixes at least two points of the quotient curve $\mathcal{Z}=\cC/Q$ that will be investigated in Section \ref{secquotient}.
}}
\end{remark}
\section{Some quotient curves of the DGZ-curve}
\label{secquotient}
Let $\cY=\cC/T$ be the quotient curve of the DGZ curve with respect to the group of translations $T$.
\begin{prop}
\label{prop8Feb}
Let $\xi=x^q-x$ and $\eta=y^q-y$. Then the function field of $\mathcal{Y}$ is $\K(\xi,\eta)$ with $H(\xi,\eta)=0$ where $H(X,Y)\in \F[X,Y]$ is the absolutely irreducible polynomial such that
\begin{equation}
\label{eq8feb} H(X,Y)=\frac{X^{q^2-1}-Y^{q^2-1}}{X^{q-1}-Y^{q-1}}+1.
\end{equation}
\end{prop}
\begin{proof} Since $\varphi_{\alpha,\beta}(\xi)=\varphi_{\alpha,\beta}(x^q)-\varphi_{\alpha,\beta}(x)=x^q+\alpha^q-(x+\alpha)=x^q-x=\xi$, $\K(\cY)$ contains $\xi$. Similarly, $\eta\in \K(\cY)$. Therefore, $\K(\xi,\eta)$ is a subfield of $\K(\cY)$. On the other hand
$[\K(\cC):\K(\cY)]=q^2.$ As $[\K(\cC):\K(\xi,y)]=q$ and $[\K(\xi,y):\K(\xi,\eta)]=q$, this shows that $[\K(\cC):\K(\xi,\eta)]=q^2$, whence $\K(\xi,\eta)=\K(\mathcal{Y})$ follows. Since $q-1$ divides $q^2-1$, $H(X,Y)$ is a polynomial whose degree equals $q^2-q$. Furthermore,
$$F(x,y,1)=\frac{D_1(x,y,1)}{D_2(x,y,1)}=
\frac{(\xi^{q^2}+\xi^q+\xi)(\eta^{q^2}+\eta^q)-
(\eta^{q^2}+\eta^q+\eta)(\xi^{q^2}+\xi^q)}{(\xi^q+\xi)\eta^q-(\eta^q+\eta)\xi^{q}}\,.$$
Observe that right hand side can also be written as
$$\frac{\xi^{q^2-1}-\eta^{q^2-1}+\xi^{q-1}-\eta^{q-1}}{\xi^{q-1}-\eta^{q-1}}\,.$$
 As $F(x,y,1)=0$, this yields $H(\xi,\eta)=0$. Take an absolutely irreducible factor $L(X,Y)$ of $H(X,Y)$ so that $L(\xi,\eta)=0$. Then the polynomial $M(X,Y)=L(X^q-X,Y^q-Y)$ has degree at most $q^3-q^2$, and $M(x,y)=0$. As  $F(x,y)=0$ this yields that $\deg\,F(X,Y)\leq \deg M(X,Y)$ whence $\deg M(X,Y)=q^3-q^2$ follows. Hence, $\deg H(X,Y)=\deg L(X,Y)$ showing that  $H(X,Y)$ is absolutely irreducible.
\end{proof}
\begin{prop}
\label{prop5feb2017} Let $P$ be a point of $\cX$ where the cover $\cX|\cY$ ramifies. Then the second ramification
group at $P$ is trivial.
\end{prop}
\begin{proof} In terms of $\cC$, the ramification points of the cover $\cX|\cY$ are the fixed points of $T$. They are exactly the points of $\cC$ lying on the line $z=0$ at infinity. Therefore, $P=(1:b:0)$ with $b\in \FF\setminus \F$. As in the proof of Proposition \ref{propB30gen}, $P$ is taken to a point $R=(0:b:1)$
by the linear map $\sigma$ with matrix representation $$\Sigma=\begin{pmatrix} 0 & 0 & 1 \\ 0 & 1 & 0 \\ 1 & 0 & 0 \end{pmatrix}. $$
The image $\sigma(\cC)$ is an algebraic plane curve $\mathcal{D}$ isomorphic to $\cC$.
Furthermore, $\varphi_{\alpha,\beta}$, regarded as a linear map preserving $\cC$, has matrix representation
$$\Lambda = \begin{pmatrix} 1 & 0 & \alpha \\ 0 & 1 & \beta \\ 0 & 0 & 1 \end{pmatrix} .$$
 Since $$\Sigma^{-1}\Lambda\Sigma=\begin{pmatrix} 1 & 0 & 0 \\ \beta & 1 & 0 \\ \alpha & 0 & 1 \end{pmatrix},$$  the map
 $$\tilde{\varphi}_{\alpha,\beta}:(x,y)\mapsto
\left(\frac{x}{\alpha x+1},\frac{y+\beta x}{\alpha x+1}\right)$$ is an automorphism of $\mathcal{D}$.
 Therefore, $R$ is a point of the algebraic plane curve $\mathcal{D}=\sigma(\cC)$, and the maps $\tilde{\varphi}_{\alpha,\beta}$ with $\alpha,\beta \in \F$ form a subgroup $\tilde{T}$ of $\aut(\mathcal{D})$. We have to determine the second ramification group of $\tilde{T}$ at $R$. For this purpose, take $\bar{t}=\frac{x}{y-b}$ for a local parameter at $R$, and  compute $v_{R}(\tilde{\varphi}_{\alpha,\beta}(\bar{t})-\bar{t})$. A direct computation yields $$\tilde{\varphi}_{\alpha,\beta}(\bar{t})-\bar{t}=\tilde{\varphi}_{\alpha,\beta}\left(\frac{x}{y-b}\right)-\frac{x}{y-b}=\frac{x}{\alpha x+1}\cdot\frac{\alpha x+1}{\beta x+y-b(\alpha x+1)}-\frac{x}{y-b}=$$$$=x^2 \cdot \frac{(b\alpha-\beta)}{(x(\beta-b\alpha)+y-b)(y-b)}.$$ From (\ref{branchP}) and Corollary \ref{cor31gen2018}, $\mathcal{D}$ has a unique branch $\bar{\gamma}$ centred at $R$ with primitive representation $$\begin{cases} x(t)=c_2t^{q} \\ y(t)=b+d_2t^{q-1}+\ldots \end{cases},$$ then
$$\tilde{\varphi}_{\alpha,\beta}(\bar{t})-\bar{t}=\frac{c_2t^{2q}(b\alpha-\beta)}{(c_2t^q(\beta-b\alpha)+d_2t^{q-1})t^{q-1}+\ldots}=t^2+ h(t)$$ with $\text{ord}(h)>2$. So $v_{R}(\tilde{\varphi}_{\alpha,\beta}(\bar{t})-\bar{t})=2$ and the assertion follows.
\end{proof}
\begin{prop}
The curve $\cY$ has genus $\gg(\cY)=\frac{1}{2}(q-1)(q^2-2q-2)+1$.
\end{prop}
\begin{proof}From the proof of Proposition \ref{prop5feb2017}, the cover $\cX|\mathcal{Y}$ ramifies at the points of $\cC$ on the line $\ell_\infty$. This, together with Proposition \ref{prop5feb2017}, gives
$$\sum \limits_{P\in \ell_{\infty}}d_P=\sum \limits_{P\in \ell_{\infty}} (\vert G_P^{(i)} \vert -1)=\sum \limits_{P\in \ell_{\infty}}((\vert G_P^{(0)} \vert -1)+(\vert G_P^{(1)} \vert -1))=2(q^2-1)(q^2-q).$$
Now, from the Hurwitz genus formula, the assertion follows.
\end{proof}
Since the normalizer of $T$ is larger than $T$, some more quotient curves of $\cY$ (and hence of $\cX$) arise.

As pointed out in Section \ref{secaut}, a subgroup of $\text{Aut}(\mathcal{Y})$ of order $q$ is $\Phi :=\lbrace \phi_\gamma \mid \gamma\in \mathbb{F}_q \rbrace$, where
$$
\phi_\gamma:(\xi,\eta)\mapsto (\xi,\gamma\xi+\eta).
$$
Let $\cZ=\mathcal{Y}/\Phi$ be the quotient curve of $\mathcal{Y}$ with respect to $\Phi$. To find an equation
for $\cZ$ it is useful to represent $\K(\mathcal{Y})$ in the form $\K(v,w)$ with $v=\eta \xi^{-1}$ and $w=\xi^{-1}$, equivalently, $\xi=w^{-1}$ and $\eta=vw^{-1}$. From $H(\xi,\eta)=0$, we have $M(v,w)=0$ where $M(X,Y)$ is the absolutely irreducible polynomial defined by
\begin{equation} \label{eqc2} M(X,Y)=\frac{X-X^{q^2}}{X-X^q}+Y^{q^2-q}=1+X^{q-1}-X^{q(q-1)}+Y^{q(q-1)}.
\end{equation}
This shows that $\K(\mathcal{Y})=\K(v,w)$ with $M(v,w)=0$. With this notation,
$$
\phi_\gamma:(v,w)\mapsto (v+\gamma,w).
$$
\begin{prop}
Let $\theta:=v-v^q$ and $\sigma:=w$. Then the function field of $\mathcal{Z}$ coincides with the function field $\mathbb{K}(\theta,\sigma)$ defined by
\begin{equation}
\label{eqc3} 1+\theta^{q-1}+\sigma^{q(q-1)}=0,
\end{equation}
that is $\mathcal{Z}$ has equation $R(X,Y)=1+X^{q-1}+Y^{q(q-1)}=0.$
\end{prop}
\begin{proof}
As in the proof of Proposition \ref{prop8Feb}, $\phi_\gamma(\theta)=\theta$ and $\phi_\gamma(\sigma)=\sigma$ shows that $\K(\theta,\sigma)$ is a subfield of $\K(\cZ)$. On the other hand $[\K(\cY):\K(\cZ)]=q.$ Since $[\K(\cY):\K(\theta,w)]=q$ and $\K(\theta,w)=\K(\theta,\sigma)$, then $[\K(\cY):\K(\theta,\sigma)]=q$, whence $\K(\theta,\sigma)=\K(\mathcal{Z})$ follows. On the other hand,$$R(\theta,\sigma)= 1+\theta^{q-1}+\sigma^{q(q-1)}=1+v^{q-1}-v^{q(q-1)}+w^{q(q-1)}=M(v,w)=0.$$
\end{proof}
\begin{prop}
\label{prop8feb} The curve $\mathcal{Z}$ is isomorphic to the Fermat curve of degree $q-1$.
\end{prop}
\begin{proof} Let $\K(\mathcal{F}_{q-1})=\K(s,t)$ with $s^{q-1}+t^{q-1}+1=0$ be the function field of the Fermat curve $\mathcal{F}_{q-1}$ of degree $q-1$. Then the map
\begin{equation*}
\begin{matrix} \K(\mathcal{F}_{q-1})\to \K(\mathcal{Z}) \\ (s,t)\to (s^q,t) \end{matrix}
\end{equation*}
is an isomorphism from $\mathcal{F}_{q-1}$ to $\cZ$, since $R(s^q,t)=1+(s^q)^{q-1}+t^{q(q-1)}=(1+s^{q-1}+t^{q-1})^q=0.$
\end{proof}
\begin{remark}
\label{remark1Mar} {\rm{Proposition \ref{prop8feb} shows that $\mathcal{Z}$ is a line for $q=2$ while it is an irreducible conic for $q=3$.}}
\end{remark}
\begin{remark}
\label{rem27mar2018} {\em{Let $\gamma({\mathcal{F}}_{q-1})$ denote the $p$-rank of ${\mathcal{F}}_{q-1}$. Since $Q$ leaves $\ell_\infty$ invariant and $T$ is the subgroup of $Q$ fixing $\ell_\infty$
pointwise, each of the $q^2-q$ common points of $\cC$ with $\ell_\infty$ is fixed by exactly $q^2$ elements of $Q$. Furthermore, each point on a line of $PG(2,\mathbb{F}_q)$ through $P_\infty$ is fixed by exactly $q$ elements of $Q$, and hence each of the $q(q^2-q)$ common points of $\cC$ with such lines is fixed by exactly $q$ elements of $Q$. No more point of $\cC$ is fixed by some nontrivial element of $Q$. From Proposition \ref{prop8feb} and Deuring Shafarevic formula applied to $Q$,
$$\gamma(\cC)=q^3(\gamma({\mathcal{F}}_{q-1})-1)+(q^2-q)(q^2-1)+q(q^2-q)(q-1)+1.$$
This shows that $\gamma({\mathcal{F}}_{q-1})$ determines $\gamma(\cC)$ and viceversa. Unfortunately, $\gamma({\mathcal{F}}_{q-1})$ is only known for $q=p$ in which case ${\mathcal{F}}_{p-1}$ is ordinary and hence
$\gamma({\mathcal{F}}_{p-1})=\frac{1}{2}(p-2)(p-3)$; see \cite{yn}. In this case $\gamma(\cC)=\frac{1}{2}p(p-1)(p^3-2p-2)+1$ which shows that $\cC$ is ordinary, as well. For $q>p$, ${\mathcal{F}}_{q-1}$ and hence $\cC$ are not ordinary although both have positive $p$-rank; see \cite{yn}.}}
  \end{remark}
\section{The full automorphism group of the DGZ curve}
\label{secfullaut}
This section is devoted to a deeper investigation of the automorphism group of $\cC$.
By Lemma \ref{lemA30gen2018}, $\aut(\cC)$ contains a subgroup $G\cong PGL(3,\F)$. Our goal is to prove that
$\aut(\cC)=G$. This result is true for $q=2$; see Remark \ref{remA16mar2018}.
\begin{prop}
\label{prop7Mar} $Q$ is a Sylow $p$-subgroup of $\aut(\cC)$.
\end{prop}
\begin{proof} We adopt notation and hypotheses from Remark \ref{rem15mar2018}.


From Proposition \ref{prop8feb}, $\bar{U}$ can be regarded as a $p$-subgroup of $\aut(\mathcal{F}_{q-1})$. As $|\aut(\mathcal{F}_{q-1})|=6(q-1)^2$, this is only possible for $p\leq 3$.

 Let $p=3$. Then a Sylow $p$-subgroup of $\aut(\mathcal{F}_{q-1})$ has order $3$ and its nontrivial elements are $\sigma$ and $\sigma^2$ where $\sigma(x,y,z)=(y,z,x)$. Since $p=3$, $\sigma$ (and also $\sigma^2$) viewed as an automorphism of $PG(2,\K)$ has a unique fixed point, namely $(1:1:1)$ which is not a point of $\mathcal{F}_{q-1}$. Actually, this holds true for $\bar{U}$ as $\bar{U}$ is also a Sylow $p$-subgroup of $\aut(\mathcal{F}_{q-1})$.
 On the other hand, from Remark \ref{rem15mar2018}, $\bar{U}$ has a fixed point in $\mathcal{F}_{q-1}$. This contradiction rules out the case $p=3$.

Let $p=2$. This time a Sylow $p$-subgroup of $\aut(\mathcal{F}_{q-1})$ has order $2$ and its nontrivial element is
$\sigma(x,y,z)=(y,x,z)$. In particular, $\sigma$ viewed as an automorphism of $PG(2,\K)$ fixes the line $x=y$ pointwise but no further point. Just one fixed point of $\sigma$, namely $(1,1,0)$ is in $\mathcal{F}_{q-1}$. This remains true for $\bar{U}$ as $\bar{U}$ is also a Sylow $p$-subgroup $\aut(\mathcal{F}_{q-1})$. On the other hand, from Remark \ref{rem15mar2018}, $\bar{U}$ has at least two fixed points in $\mathcal{F}_{q-1}$. This contradiction rules out the case $p=2$.
\end{proof}

\begin{prop}
\label{Orbite21Feb}
$\aut(\cC)$ has exactly two short orbits $\Omega$ and $\Delta$ where $\Omega$ is non-tame and $\Delta$ is tame.
\end{prop}
\begin{proof}
From the proof of Proposition \ref{propB30gen}, the set of $q^4-q$ points lying in $PG(2,\mathbb{F}_{q^2})\setminus PG(2,\mathbb{F}_q)$ is an orbit under the action of $PGL(3,q)$ and the stabiliser in $PGL(3,q)$ of any such point has order $q^2(q^2-1)$ by Result \ref{resorbit}. Therefore $\cC$ has a non-tame short orbit $\Omega$ under the action of $\aut(\cC)$.

We show that $\aut(\cC)$ has a short tame orbit $\Delta$, as well. From Lemma \ref{le14mar2018}, $\cC(\mathbb{F}_{q^3})$ consists of all points in $\Lambda_3$. Take $P\in \cC(\mathbb{F}_{q^3})$ and assume by contradiction that there exists a group $D\leq \aut(\cC)_P$ such that $|D|=p^l$ for some $l\geq 1$.
Let $S_p$ a Sylow $p$-subgroup of $\aut(\cC)$ containing $D$. Then $S_p$ and $Q$ are conjugate in $\aut(\cC)$, that is, $Q=\gamma S_p \gamma^{-1}$ for some $\gamma \in \aut(\cC)$. Let $U=\gamma D \gamma^{-1}$ and $P'=\gamma(P)$. Then $U(P')=P'$. As no element in $Q\setminus T$ fixes a point in $\cC$, $U$ is a subgroup of $T$ and $P'\in \cC(\mathbb{F}_{q^2})$. Also, the set of all fixed points of $T$ in $\cC(\mathbb{F}_{q^2})$ has size $q^2-q$. Let $M=\gamma^{-1}T\gamma$. Then $M$ also has exactly $q^2-q$ fixed points, and the stabiliser of $P$ in $S_p$ coincides with $M$. Let $C_{q^2+q+1}$ be the Singer subgroup of $\aut(\cC)$ fixing $P$. Then $C_{q^2+q+1}$ fixes exactly two
other points in $PG(2,\mathbb F_{q^3})$ and no non-trivial element of $C_{q^2+q+1}$ fixes a further point.
From Result \ref{res74}, $C_{q^2+q+1}$ normalizes $M$. Therefore, the set of fixed points of $M$ is left invariant by $C_{q^2+q+1}$. Since $q^2+q+1>q^2-q>3$, this yields that some nontrivial element in $C_{q^2+q+1}$ has more than three fixed points, a contradiction.

Therefore, $\Omega$ and $\Delta$ are short orbits of $\aut(\cC)$. From Result \ref{res56.116}, either they are the only short orbits of $\aut(\cC)$, or there exists just one further tame short orbit of $\aut(\cC)$. The latter possibility may be investigated as in the proof of Result \ref{res56.116}. From (III) in that proof, this possibility may only occur when $p>2$  and $\vert \aut(\cC)_P \vert=2$ for the stabiliser of each point $P$ in the tame orbits. But in our case, if $P\in \cC(\mathbb{F}_{q^3})$, the Singer subgroup of $\aut(\cC)$ fixing $P$ has order $q^2+q+1>2$; a contradiction.
\end{proof}

\begin{teorema}
$\aut(\cC)\cong PGL(3,q)$.
\end{teorema}
\begin{proof}
Take a point $P\in\Omega$ and a point $R\in \Delta$.
Proposition \ref{Orbite21Feb} together with the Hurwitz genus formula applied to $\Gamma=\aut(\cC)$ give
\begin{equation}
\label{eq8mar2018}
2\gg(\cC)-2=\vert\Gamma \vert (2\bar{\gg}-2)+\vert \Omega \vert d_P + \vert \Delta \vert d_R
\end{equation}
 where $\bar{\gg}=\gg(\mathcal{X}/\Gamma)$.
Since $\vert \Gamma\vert >84(\gg(\cC)-1)$ for $q>2$, Result \ref{res56.116} yields $\bar{\gg}=0$. Furthermore, $\vert \Gamma_P\vert\geq q^2(q^2-1)$ and $\vert \Gamma_P^{(1)}\vert=q^2$. Actually, $\vert \Gamma_P\vert= q^2(q^2-1)$ and $\Gamma_P^{(2)}$ is trivial by Result \ref{resnaka}. Then
(\ref{eq8mar2018}) reads
\begin{equation} \label{eq18Feb}
 2\gg(\cC)-2=-2\vert\Gamma \vert+\frac{\vert \Gamma \vert}{\vert \Gamma_P \vert} (\vert \Gamma_P \vert-1+\vert \Gamma_P^{(1)}\vert-1) + \frac{\vert \Gamma \vert}{\vert \Gamma_R\vert} d_R
\end{equation}

As $\Gamma_R$ does not contain $p$-elements, $d_R=(\vert \Gamma_R\vert-1)$ and (\ref{eq18Feb}) reads
\begin{equation*}
2\gg(\cC)-2=\frac{\vert \Gamma \vert}{\vert \Gamma_P\vert}\left( q^2\left(1-\frac{q^2-1}{\vert \Gamma_R \vert}\right)-2\right).
\end{equation*}
Furthermore, $\vert \Gamma_R \vert = \lambda(q^2+q+1)$ with $\lambda \geq 1$ as $R$ is a fixed point of a Singer subgroup of $PGL(3,q)$. Therefore,
\begin{equation*}
1-\frac{q^2-1}{\vert \Gamma_R \vert}=\frac{(\lambda-1)q^2+\lambda q+\lambda+1}{\lambda(q^2+q+1)}>\frac{\lambda-1}{\lambda}.
\end{equation*}
Suppose $\lambda>1$. From (\ref{eq18Feb}),
\begin{equation*}
2\gg(\cC)-2=\frac{\vert \Gamma \vert}{\vert \Gamma_P\vert}\left( q^2\left(1-\frac{q^2-1}{\vert \Gamma_R \vert}\right)-2\right)>\frac{\vert \Gamma \vert}{\vert \Gamma_P\vert}\left(\frac{q^2}{2}-2\right),
\end{equation*}
that is impossible since $\vert \Gamma \vert \geq q^3(q^3-1)(q^2-1)$. Hence $\lambda=1$ and $\vert \Gamma_R\vert=q^2+q+1$.
Finally, (\ref{eq18Feb}) yields $$\vert \Gamma \vert=q^3(q^3-1)(q^2-1)=\vert PGL(3,q)\vert.$$

\end{proof}

\section{The Geometry of the DGZ-curve}
\label{sectionnc}
We show that $\cC$ has exceptional geometric properties, as well.
\begin{prop} The DGZ curve is non-classical and $\mathbb{F}_q$-Frobenius non-classical.
\label{propA19mar2018}
\end{prop}
\begin{proof} Let $$G_1(x,y,z)=yz^{q^2}-y^{q^2}z,\, G_2(x,y,z)=x^{q^2}z-xz^{q^2},\, G_0(x,y,z)=xy^{q^2}-x^{q^2}y.$$
A straightforward computation shows that
Equation (\ref{eqA30gen2018}) can also be written as $D_2F=G_1^qx+G_2^qy+G_0^qz.$
By (\ref{eq19mar2018}) with $H=D_2$ and $q=p^m$, $\cC$ is non-classical. Furthermore, $G_1x+G_2y+G_0z$ is the zero polynomial. This shows that $\cC$ is $q$-Frobenius non-classical.
\end{proof}
Obviously, $\cC$ may be regarded as a curve defined over $\mathbb{F}_{q^i}$ with $i\geq 1$. For $q=2$, $\cC$ has equation (\ref{eqC19mar2018}). Then \cite[Section 3]{jt} yields that $\cC$ is $\mathbb{F}_8$-Frobenius non-classical. We prove that Top's result holds for any $q$.
\begin{prop} The DGZ curve is $\mathbb{F}_{q^3}$-Frobenius non-classical.
\label{propB19mar2018}
\end{prop}
\begin{proof} From Proposition \ref{propA19mar2018}, $\cC$ is non-classical. A straightforward computation shows that $G_1x^{q^2}+G_2y^{q^2}+G_0z^{q^2}$ is the zero polynomial so the assertion follows from equation (\ref{eqA19mar2018}).
\end{proof}

Let $\ell$ be a line of $PG(2,\mathbb{F}_{q})$. From Lemma \ref{lem31jan2018}, the intersection divisor of $\cC$ cut out by $\ell$ is $$V= \sum_{i=1}^{q^2-q} qP_i$$
where $P_1,\ldots,P_{q^2-q}$ are the common points of $\cC$ and $\ell$.
Therefore, the ramification divisor of $\mathcal{L}$ is
$$R={\rm{div}}(W_R)+(q+1){\rm{div}}(dx)+3V,$$
and $\deg(R)=(q+1)(2\gg(\cC)-2)+3(q^3-q^2)=q(q-1)(q^4+q^3-2q^2-q-2)$. Furthermore, Proposition \ref{propA19mar2018} yields $(\varepsilon_0,\varepsilon_1,\varepsilon_2)=(0,1,q)$.

In terms of $(\mathcal{L},P)$-orders, Lemma \ref{lem31jan2018}
is stated in the following lemma.
\begin{lem}
\label{lemA20mar2018}
For $P\in \cC(\mathbb{F}_{q^2})$, the $({\mathcal{L}},P)$-order sequence is $(0,q-1,q)$, and $v_P(R)=q-2$.
\end{lem}
\begin{proof}
Let $P\in \cC(\mathbb{F}_{q^2})$. Then the $({\mathcal{L}},P)$-order sequence is $(0,q-1,q)$ as a consequence of Lemma \ref{lem31jan2018}.
Furthermore, observe that the matrix $\left(\left(\begin{array}{c}
 j_i \\
\varepsilon_k
\end{array} \right)\right)$
has determinant $q-1 \not\equiv 0 \pmod{p}$. Therefore $v_P(R)=q-2$ from \cite[Theorem 7.55]{HKT}.
\end{proof}
\begin{lem}
\label{lemB20mar2018}
For a point $P\not\in \cC(\mathbb{F}_{q^2})\cup  \cC(\mathbb{F}_{q^3})$, the $({\mathcal{L}},P)$-order sequence is $(0,1,q)$, and $v_P(R)=0$.
\end{lem}
\begin{proof}
Assume on the contrary that $v_P(R)=m>0$ for some point $P\in \Gamma$ with $\Gamma=\cC\setminus(\cC(\mathbb{F}_{q^2})\cup  \cC(\mathbb{F}_{q^3}))$.
Since $\cC(\mathbb{F}_{q})=\emptyset$ by Proposition \ref{propA30gen}, the orbit of $P$ in $\aut(\cC)$ is long by Proposition \ref{Orbite21Feb}. Therefore
$$\sum_{P\in \Gamma} v_P(R)=m|PGL(3,q)|=mq^3(q^3-1)(q^2-1).$$
But this contradicts $\deg(R)=q(q-1)(q^4+q^3-2q^2-q-2)$. Then $v_P(R)=0$ for any point $P\not\in \cC(\mathbb{F}_{q^2})\cup  \cC(\mathbb{F}_{q^3})$ and the $({\mathcal{L}},P)$-order sequence is $(0,1,q)$.
\end{proof}
\begin{lem}
\label{lemD20mar2018}
For $P\in \cC(\mathbb{F}_{q^3})$, the $({\mathcal{L}},P)$-order sequence is $(0,1,q+1)$, and $v_P(R)=1$.
\end{lem}
\begin{proof} Let $v_P(R)=m$ for a point $P\in \cC(\mathbb{F}_{q^3})$. Since $\cC(\mathbb{F}_{q^3})$ is an orbit of
$\aut(\cC)$, we have $v_P(R)=m$ for every $P\in \cC(\mathbb{F}_{q^3})$. From Lemmas \ref{lemA20mar2018} and \ref{lemB20mar2018}, $q(q-1)(q^4+q^3-2q^2-q-2)=\deg(R)=(q-2)(q^4-q)+m(q^6-q^5-q^4+q^3),$ whence $m=1$. In particular, $j_2\geq q+1$. On the other hand, if $j_2>q+1$ then $v_P(R)\ge \sum_{i=0}^2 (j_i-\varepsilon_i)>1$ by \cite[Theorem 7.55]{HKT}. Therefore, $j_2=q+1$.
\end{proof}
As a corollary we have the following result.
\begin{prop}
\label{lemA21mar2018}
$$
v_P(R)=
\begin{cases}
q-2, & \mbox{for } P\in \cC(\mathbb{F}_{q^2}),\\
1, & \mbox{for } P\in \cC(\mathbb{F}_{q^3}),\\
 0, & \mbox{otherwise }.
\end{cases}
$$
\end{prop}

Since $\mathcal{L}$ is defined over $\mathbb{F}_q$, $\cC$ also has its $\mathbb{F}_q$-Frobenius order sequence $(\nu_0,\nu_1)$. In our case $\nu_0=0$ and $\nu_1=q$ by Proposition \ref{propA19mar2018}, and the St\"ohr-Voloch divisor of $\mathcal{L}$ over $\mathbb{F}_q$ is
 $$S={\rm{div}}(W_S)+q{\rm{div}}(dx)+(q+2)V.$$
Thus $\deg(S)=q(2\gg(\cC)-2)+(q+2)(q^3-q^2)=q^6-q^5-q^4+q^3$.
\begin{prop}
\label{prop21mar2018}
$$
v_P(S)=
\begin{cases}
1, & \mbox{for } P\in \cC(\mathbb{F}_{q^3}),\\
 0, & \mbox{otherwise. }
\end{cases}
$$
\end{prop}
\begin{proof}
Replacing $R$ with $S$ in the proof of Lemma \ref{lemB20mar2018} we see that $v_P(S)=0$ for $P\in \cC\setminus(\cC(\mathbb{F}_{q^2})\cup  \cC(\mathbb{F}_{q^3}))$.
Let $m_1=v_P(S)$ for $P\in \cC(\mathbb{F}_{q^3})$ and $m_2=v_P(S)$ for $P\in \cC(\mathbb{F}_{q^2})$.
Then $q^6-q^5-q^4+q^3=\deg(S)=m_1(q^6-q^5-q^4+q^3)+m_2(q^4-q)$, whence  $m_1=1$ and $m_2=0$ follow.
\end{proof}
\begin{lem}
\label{le22mar2018} Let $\cD$ be a plane absolutely irreducible curve defined over $\mathbb{F}_{q^3}$ which is
non-classical with order sequence $(0,1,p^\alpha q)$, $\alpha\geq 0$. Then $\alpha=0$ if $\cD$ has the following properties:
\begin{itemize}
\item[\rm(i)] $\deg(\cD)=\deg(\cC)=q^3-q^2$,
\item[\rm(ii)] $|\mathbb{F}_{q^3}(\cU)|\geq |\mathbb{F}_{q^3}(\cC)|=q^6-q^5-q^4+q^3$, where $\cU$ is a nonsingular model of $\cD$ defined over $\mathbb{F}_{q^3}$.
\end{itemize}
 \end{lem}
\begin{proof} Let $\Omega$ denote the set of all branches of $\cD$ which correspond to the $\mathbb{F}_{q^3}$-rational points of $\cU$. Each branch  $\gamma \in\Omega$ is centered at a point in $PG(2,\mathbb{F}_{q^3})$. For any line $\ell$ of $PG(2,\mathbb{F}_{q^3})$, a pair $(\gamma,\ell)$ is called {\emph{incident}} if the center of $\gamma$ lies on $\ell$.

Obviously, we have as many as $|\Omega|(q^3+1)$ incident branch-line pairs $(\gamma,\ell)$.
 On the other hand, for any line $\ell$ of  $PG(2,\mathbb{F}_{q^3})$, let $\lambda(\ell)$ denote the number of branches in $\Omega$ whose center lies on $\ell$.
The double counting of such incident branch-line pairs gives
\begin{equation}
\label{eq13mar2018}
|\Omega|(q^3+1)=\sum_{\ell\in PG(2,\mathbb{F}_{q^3})} \lambda(\ell).
\end{equation}
It is useful to divide the lines of $PG(2,\mathbb{F}_{q^3})$ into two families: $\Sigma_1$ comprises all lines $\ell$ which are tangent to some branches in $\Omega$, while $\Sigma_2$ consists of the remaining lines.
 For a line $\ell\in \Sigma_1$ which is tangent to $\gamma\in \Omega$, the intersection number $I(\gamma,\cD\cap \ell)\ge qp^\alpha$. Therefore, if $\gamma_1,\ldots,\gamma_{r_\ell}\in \Omega$ are the branches in $\Omega$ tangent to $\ell$ then B\'ezout's theorem yields $${\mbox{$\lambda(\ell)=|\cD\cap\ell|\le (q^3-q^2)-r_\ell qp^{\alpha}$, for $\ell\in \Sigma_1$.}}$$
Since each $\gamma$ has a unique tangent, we have $\sum_{\ell\in \Sigma_1}r_{\ell}=|\Omega|$. Furthermore, if $\ell\in \Sigma_2$ then the obvious upper bound on $\lambda(\ell)$ is $q^3-q^2$.
From (\ref{eq13mar2018}),
$$|\Omega|(q^3+1)+|\Omega|p^\alpha q \le (q^6+q^3+1)(q^3-q^2).$$
This together with (ii) give
$$(q^6-q^5-q^4+q^3)(q^3+1+p^\alpha q)\leq (q^6+q^3+1)(q^3-q^2),$$
whence $\alpha=0$ follows.
\end{proof}
\begin{prop}
\label{pro22mar2018} Let $\cD$ be a plane absolutely irreducible curve defined over $\mathbb{F}_{q^3}$ such that
\begin{itemize}
\item[\rm(I)] $\deg(\cD)=\deg(\cC)=q^3-q^2$,
\item[\rm(II)]$\gg(\cD)=\gg(\cC)=\frac{1}{2}q(q-1)(q^3-2q-2)+1$,
\end{itemize}
Then $|\mathbb{F}_{q^3}(\cU)|\leq |\mathbb{F}_{q^3}(\cC)|$, where $\cU$ is a nonsingular model of $\cD$ defined over $\mathbb{F}_{q^3}$.
 \end{prop}
\begin{proof} Let $\mathcal{L}'$ be the $2$-dimensional linear series cut out on $\cD$ by lines. Then the St\"ohr-Voloch divisor $S'$ of $\mathcal{L}$ over $\mathbb{F}_{q^3}$ has degree
$$\deg(S')=\nu'(2\gg(\cD)-2)+(q^3+2)(q^3-q^2)$$
where $\nu'$ is the $\mathbb{F}_{q^3}$-Frobenius order of $\cD$. Assume that $|\mathbb{F}_{q^3}(\cU)|\geq |\mathbb{F}_{q^3}(\cC)|$. Then either $\cD$ is $\mathbb{F}_{q^3}$-Frobenius classical, or Lemma \ref{le22mar2018} yields $\nu'\le q$. In both cases $1\leq \nu'\leq q$. This together with (I) and (II) yield $\deg(S')\leq \deg(S)$.
Therefore,
\begin{equation}
\label{eqB23mar2018}
\sum_{P'\in \mathbb{F}_{q^3}(\cU)}v_P(S')\leq \sum_{P\in \mathbb{F}_{q^3}(\cC)}v_P(S).
\end{equation}
Since $v_{P'}(S')\geq 1$ for any $P'\in \mathbb{F}_{q^3}(\cU)$ while $v_{P}(S)=1$ for any $P\in \mathbb{F}_{q^3}(\cC)$ by Proposition \ref{prop21mar2018}, (\ref{eqB23mar2018}) yields $|\mathbb{F}_{q^3}(\cU)|=|\mathbb{F}_{q^3}(\cC)|$.
\end{proof}
Proposition \ref{pro22mar2018} has the following corollary.
\begin{cor}
\label{cor24mar2018}
The DGZ curve is a $(q^3-q^2,\frac{1}{2}q(q-1)(q^3-2q-2)+1,3)$ optimal curve over $\mathbb{F}_{q^3}$.
\end{cor}
\begin{remark}
{\emph{Proposition \ref{pro22mar2018} remains valid if $\mathbb{F}_{q^3}(\cU)$ is replaced by the set $\mathbb{F}_{q^3}(\cD)$ of
all points of $\cD$ lying on $PG(2,\mathbb{F}_{q^3})$. In fact, the proof of Proposition \ref{pro22mar2018} still works whenever $\Omega$ stands for $\mathbb{F}_{q^3}(\cD)$ and $\sum_{\ell\in \Sigma_1}r_{\ell}=|\Omega|$ is replaced by $\sum_{\ell\in \Sigma_1}r_{\ell}\ge |\Omega|$.}}
\end{remark}
Finally we point out a combinatorial property of $\cC(\mathbb{F}_{q^3})$. For this purpose, recall that a $(k,n)$-arc $\mathcal{K}$ in the projective plane $\Pi$ consists of $k$ points in $\Pi$ such that some line in $\Pi$ meets $\mathcal{K}$ in exactly $n$ pairwise distinct points but no line in $\Pi$ meets $\mathcal{K}$ in more than $n+1$ points. Furthermore, $\mathcal{K}$ is called complete, that is, it is maximal, if no point $P\in\Pi$ other than those in $\mathcal{K}$ exists such that $\mathcal{K}\cup \{P\}$ is a $(k+1,n)$-arc. Complete $(k,n)$-arcs, especially $(k,2)$-arcs, have  intensively been investigated in Finite geometry, and they have relevant applications in Coding theory; see \cite{HJ} and \cite[Chapter 13]{HKT}. In that context, an interesting problem is to find plane curves $\cD$ defined over a finite field $\mathbb{F}$ whose set of points in $PG(2,\mathbb{F})$ is a complete $(k,n)$-arc. It seems plausible that only a few curves with such combinatorial property may exist, see \cite{Giulietti}. Our contribution in this direction is the following result. 
\begin{prop}
The set $\cC(\mathbb{F}_{q^3})$ is a complete $(q^6-q^5-q^4+q^3,q^3-q^2)$-arc.
\end{prop}
\begin{proof}
From a combinatorial point of view, $\cC(\mathbb{F}_{q^3})$ consists of all points in $PG(2,\mathbb{F}_{q^3})$ which 
are uncovered by lines defined over $\mathbb{F}_{q}$. Through a point $P$ in $PG(2,\mathbb{F}_{q})$, there are as many as $q^3+1$ lines defined over $\mathbb{F}_{q^3}$. Those of them which are also defined over $\mathbb{F}_q$ are $q+1$, whereas the remaining $q^3-q$ lines defined over $\mathbb{F}_{q^3}$ meet $PG(2,\mathbb{F}_q)$ only in $P$.
Choose one of these $q^3-q$ lines, say $\ell$. Then $\ell$ meets a line $r$ defined over $\mathbb{F}_{q}$ in a point distinct from $P$ if and only if $P\not \in r$. Furthermore, any two lines $r$ and $s$ both defined over $\mathbb{F}_{q}$ meet $\ell$ in two different points whenever $P\not \in r$ and $P\not\in s$. Since the number of lines defined over $\mathbb{F}_{q}$ equals $q^2+q+1$ and $q+1$ of them contain $P$, a counting argument shows that $\ell\cap\cC(\mathbb{F}_{q^3})=q^3-q^2$. On the other hand, since $\deg(\cC)=q^3-q^2$, no line meets $\cC(\mathbb{F}_{q^3})$ in more than $q^3-q^2$ points. Therefore, $\cC(\mathbb{F}_{q^3})$ is a $(k,n)$-arc in $PG(2,\mathbb{F}_{q^3})$ with $k=q^6-q^5-q^4+q^3$ and $n=q^3-q^2$. To show that such a $(k,n)$-arc is complete,
take any point $Q\in PG(2,\mathbb{F}_{q^3})\setminus \cC(\mathbb{F}_{q^3})$. Choose a point in $P\in PG(2,\mathbb{F}_{q})$ not lying on the unique line through $Q$ which is defined over $\mathbb{F}_{q}$. Then the line $\ell$ through $P$ and $Q$ is one of the $q^3-q$ lines defined over $\mathbb{F}_{q^3}$ which meet $PG(2,\mathbb{F}_q)$ only in $P$. As we have seen, $\ell$ meets $\cC(\mathbb{F}_{q^3})$ in exactly $n=q^3-q^2$ points. Since
$\ell$ also contains $Q$, this yields that $Q$ cannot be added to $\cC(\mathbb{F}_{q^3})$ in such a way that the resulting point-set $\cC(\mathbb{F}_{q^3})\cup \{Q\}$ is a $(k+1,n)$-arc. In other words, $\cC(\mathbb{F}_{q^3})$ is complete. 
\end{proof}

\end{document}